\documentclass[10pt]{article}

\usepackage{amsmath,amsbsy,amssymb,latexsym,amsthm,dsfont}
\usepackage{a4wide}

\newtheorem{thm}{Theorem}
\newtheorem{col}{Corollary}

\newtheorem{lem}{Lemma}
\newtheorem{ex}{Example}

\begin{document}

\title{Generalized Transversality Conditions\\
 in Fractional Calculus of Variations}

\author{Ricardo Almeida$^1$\\
\texttt{ricardo.almeida@ua.pt}
\and Agnieszka B. Malinowska$^2$\\
\texttt{a.malinowska@pb.edu.pl}}

\date{$^1$Department of Mathematics, University of Aveiro, 3810-193 Aveiro, Portugal\\[0.3cm]
$^2$Faculty of Computer Science, Bia{\l}ystok University of Technology,\\
15-351 Bia\l ystok, Poland}

\maketitle

\begin{abstract}
Problems of calculus of variations with variable endpoints cannot be solved without transversality conditions. Here, we establish such type of conditions for fractional variational problems with the Caputo derivative. We consider: the Bolza-type fractional variational problem, the fractional variational problem with a Lagrangian that may also depend on the unspecified end-point $\varphi(b)$, where $x=\varphi(t)$ is a given curve, and the infinite horizon fractional variational problem.
\bigskip

\noindent \textbf{Keywords}: calculus of variations; fractional calculus; fractional Euler--Lagrange equation; transversality conditions; Caputo fractional derivative.

\smallskip

\noindent \textbf{Mathematics Subject Classification}: 49K05; 26A33.
\end{abstract}


\section{Introduction}

The calculus of variations is concerned with the problem of extremizing functionals.
It has many applications in physics, geometry, engineering, dynamics, control theory, and
economics. The formulation of a problem of the
calculus of variations requires two steps: the specification of a
performance criterion; and then, the statement of physical
constraints that should be satisfied. The basic problem is stated as
follows: among all differentiable functions $x:[a,b]\to\mathbb R$
such that $x(a)=x_a$ and $x(b)=x_b$, with $x_a$, $x_b$ fixed reals,
find the ones that minimize (or maximize) the functional
$$J(x)=\int_a^b L(t,x(t),x'(t))\,dt.$$
One way to deal with this problem is to solve the second order
differential equation
$$\frac{\partial L}{\partial x}-\frac{d}{dt}\frac{\partial L}{\partial x'}=0,$$
called the Euler--Lagrange equation. The two given boundary conditions provide sufficient information to determine the two arbitrary constants. But if there are no boundary constraints, then we need to impose another conditions, called the
natural boundary conditions (see e.g. \cite{Brunt}),
\begin{equation}
\label{naturalBound}\left[\frac{\partial L}{\partial x'}\right]_{t=a}=0 \quad \mbox{ and } \quad \left[\frac{\partial L}{\partial x'}\right]_{t=b}=0.
\end{equation}
Clearly, such terminal conditions are important in models,
the optimal control or decision rules are not unique without these conditions.

Fractional calculus deals with
derivatives and integrals of a non-integer (real or complex) order. Fractional operators are non-local, therefore they are suitable for constructing models possessing memory effect. They found numerous applications in various
fields of science and engineering, as diffusion process, electrical science, electrochemistry, material creep, viscoelasticity, mechanics, control science, electromagnetic theory, \textrm{ect.} Fractional calculus is now recognized as vital mathematical tool to model the behavior
and to understand complex systems (see, \textrm{e.g.}, \cite{Cap,Das,hil,Mag,Mai,old,ort,sab}). Traditional Lagrangian and Hamiltonian mechanics cannot be used with nonconservative forces such as friction. Riewe \cite{rie} showed that fractional formalism can be used when treating dissipative problems. By inserting fractional
derivatives into the variational integrals he obtained the
respective fractional Euler--Lagrange equation, combining both conservative and nonconservative cases. Nowadays the fractional calculus of variations is a subject under strong
research. Investigations cover problems depending on Riemann-Liouville
fractional derivatives (see, \textrm{e.g.},
\cite{AGRA1,Almeida,Baleanu1,bis,Frederico1}), the Caputo fractional
derivative (see, \textrm{e.g.},
\cite{AGRA2,MyID:145,Baleanu:Agrawal,BALEANU,Frederico2,MalTor,withTatiana:Basia}), the symmetric
fractional derivative (see, \textrm{e.g.}, \cite{klimek}), the
Jumarie fractional derivative (see, \textrm{e.g.},
\cite{R:A:D:10,Jumarie,Jumarie3b}), and others
\cite{AGRA3,Almeida:AML,NunoRui,NunoRui2,Cresson,El-Nabulsi:Torres07,El-Nabulsi}.

The aim of this paper is to obtain transversality
conditions for fractional variational problems with the Caputo derivative. Namely, three types of problems are considered: the first in Bolza form, the second with a Lagrangian depending on the unspecified
end-point $\varphi(b)$, where $x=\varphi(t)$ is a given curve, and the third with infinite horizon. We note here, that from the best of our knowledge fractional variational problems with infinite horizon have not been considered yet, and this is an open research area.

The paper is organized in the following way. Section~\ref{sec2}
presents some preliminaries needed in the sequel.
Our main results are stated and proved in the remaining sections. In
Section~\ref{sec3} we consider the Bolza-type fractional variational
problem and develop the transversality conditions in a compact form. As corollaries,
we formulate conditions appropriate to various type of variable terminal points.  Section~\ref{sec4} provides
the necessary optimality conditions for fractional variational
problems with a Lagrangian that may also depend on the unspecified
end-point $\varphi(b)$, where $x=\varphi(t)$ is a given curve. Finally, in
Section~\ref{sec5} we present the transversality condition
for the infinite horizon fractional variational problem.

\section{Preliminaries}
\label{sec2}

In this section we present a short introduction to the fractional calculus, following \cite{kai,kilbas,Podlubny}. In the sequel,  $\alpha\in (0,1)$ and $\Gamma$ represents the Gamma function:
$$\Gamma(z)=\int_0^\infty t^{z-1}e^{-t}\, dt.$$

Let $f:[a,b]\rightarrow\mathbb{R}$ be a continuous function. Then,
\begin{enumerate}
\item{the left and right Riemann--Liouville fractional integrals of order $\alpha$ are defined by
$${_aI_x^\alpha}f(x)=\frac{1}{\Gamma(\alpha)}\int_a^x (x-t)^{\alpha-1}f(t)dt,$$
and
$${_xI_b^\alpha}f(x)=\frac{1}{\Gamma(\alpha)}\int_x^b(t-x)^{\alpha-1} f(t)dt,$$
respectively;}
\item{the left and right Riemann--Liouville fractional derivatives of order $\alpha$ are defined by
$${_aD_x^\alpha}f(x) =\frac{1}{\Gamma(1-\alpha)}\frac{d}{dx}\int_a^x (x-t)^{-\alpha}f(t)dt,$$
and
$${_xD_b^\alpha}f(x)=\frac{-1}{\Gamma(1-\alpha)}\frac{d}{dx}\int_x^b (t-x)^{-\alpha} f(t)dt,$$
respectively.}
\end{enumerate}

Let $f:[a,b]\rightarrow\mathbb{R}$ be a differentiable function. Then,
\begin{enumerate}
\item{the left and right Caputo fractional derivatives of order $\alpha$ are defined by
$${_a^CD_x^\alpha}f(x)= \frac{1}{\Gamma(1-\alpha)}\int_a^x (x-t)^{-\alpha}f'(t)dt,$$
and
$${_x^CD_b^\alpha}f(x)=\frac{-1}{\Gamma(1-\alpha)}\int_x^b(t-x)^{-\alpha} f'(t)dt,$$
respectively.}
\end{enumerate}
Observe that if $\alpha$ goes to $1$, then the operators ${^C_a
D_x^{\alpha}}$ and ${_a D_x^{\alpha}}$ could be replaced with
$\frac{d}{dx}$ and the operators ${^C_xD_b^\alpha}$ and
${_xD_b^\alpha}$ could be replaced
with $-\frac{d}{dx}$ (see \cite{Podlubny}). Moreover, we set ${_aI_x^0}f={_xI_b^0}f:=f$. \\
Obviously, the above defined operators are linear. If $f \in
C^1[a,b]$, then the left and right Caputo fractional derivatives of
$f$ are continuous on $[a,b]$ (cf. \cite{kilbas}, Theorem 2.2). In
the discussion to follow, we will also need the following fractional
integrations by parts (see e.g. \cite{AGRA3}):
\begin{equation}\label{IP1}\int_{a}^{b}g(x)\cdot {_a^C D_x^\alpha}f(x)dx
=\int_a^b f(x)\cdot {_x D_b^\alpha}
g(x)dx+\left[{_xI_b^{1-\alpha}}g(x) \cdot
f(x)\right]_{x=a}^{x=b}\end{equation} and
\begin{equation*}
\int_{a}^{b}g(x)\cdot {_x^C D_b^\alpha}f(x)dx =\int_a^b f(x)\cdot
{_a D_x^\alpha} g(x)dx-\left[{_aI_x^{1-\alpha}}g(x) \cdot
f(x)\right]_{x=a}^{x=b}.
\end{equation*}

Along the work, and following \cite{zbig}, we denote by
$\partial_iL$, $i=1,\ldots,m$ ($m\in \mathbb{N}$), the partial
derivative of function $L:\mathbb{R}^m\rightarrow \mathbb{R}$ with
respect to its $i$th argument. For simplicity of notation we introduce operators $[x]$ and
$\{x,\varphi\}$ defined by
$$[x](t)=(t,x(t),\, {_a^C D_t^\alpha} x(t)),$$
$$\{x,\varphi\}(t,T)=(t,x(t),\, {_a^C D_t^\alpha} x(t),\varphi(T)).$$

\section{Transversality conditions I}
\label{sec3}

Let us introduce the linear space
$(x,t)\in C^1([a,b])\times \mathbb R$
endowed with the norm
$\|(x,t)\|_{1,\infty}:=\max_{a\leq t \leq
    b}|x(t)|+\max_{a\leq t \leq
    b}\left|{_a^C D_t^\alpha} x(t)\right|+|t|$.

We consider the following type of functionals:
\begin{equation}\label{funct1}
J(x,T)=\int_a^T L[x](t)\,dt+\phi(T,x(T)),
\end{equation}
on the set $$\mathrm{D}=\left\{ (x,t)\in C^1([a,b])\times [a,b] \, | \, x(a)=x_a \right\},$$ where the {\it Lagrange function} $L:[a,b]\times \mathbb
R^2\to\mathbb R$ and the {\it terminal cost function}
$\phi:[a,b]\times \mathbb R\to\mathbb R$ are at least of class
$C^1$. Observe that we have a free end-point $T$ and no constraint
on $x(T)$. Therefore, they become a part of the optimal choice
process. We address the problem of finding a pair $(x,T)$ which
minimizes (or maximizes) the functional $J$ on $\mathrm{D}$,
\textrm{i.e.}, there exists $\delta>0$ such that $J(x,T)\leq
J(\bar{x},t)$ (or $J(x,T)\geq J(\bar{x},t)$) for all $(\bar{x},t)\in
\mathrm{D}$ with $\|(\bar{x}-x,t-T)\|_{1,\infty}<\delta$.

\begin{thm}\label{Teotransv} Consider the functional given by \eqref{funct1}.
Suppose that $(x,T)$ gives a minimum (or maximum) for functional
\eqref{funct1} on $\mathrm{D}$. Then $x$ is a solution
of the fractional Euler--Lagrange equation
\begin{equation}\label{eqeuler}
\partial_2 L[x](t)+{_t D_T^\alpha}( \partial_3L[x](t))=0
\end{equation}
 on the interval $[a,T]$ and satisfies the transversality
conditions
\begin{equation}\label{eqtransv}\left\{ \begin{array}{l}
 L[x](T)+ \partial_1\phi(T,x(T)) -x'(T) \left[ {_t I_T^{1-\alpha}}\partial_3L[x](t) \right]_{t=T}=0\\
 \left[{_t I_T^{1-\alpha}}\partial_3L[x] (t) \right]_{t=T} + \partial_2\phi(T,x(T))=0.\\
\end{array}\right.\end{equation}
\end{thm}
\begin{proof}
Let us consider a variation $(x(t)+\epsilon h(t),T+\epsilon
\triangle T)$, where $h\in C^1([a,b])$, $\triangle T \in\mathbb{R}$
and $\epsilon \in\mathbb{R}$ with $|\epsilon| \ll 1$. The constraint
 $x(a)=x_a$ implies that all admissible variations must fulfill the
condition $h(a)=0$. Define $j(\cdot)$ on a neighborhood of zero by
$$\begin{array}{ll}
j(\epsilon) & =J(x+\epsilon h,T+\epsilon \triangle T)\\
      &=\displaystyle\int_a^{T+\epsilon \triangle T} L[x+\epsilon h](t)\,dt+\phi(T+\epsilon \triangle T,(x+\epsilon h)(T+\epsilon \triangle T)).
\end{array}$$
If $(x,T)$ minimizes (or maximizes) functional \eqref{funct1} on
$\mathrm{D}$, then $j'(0)=0$. Therefore, one has
$$\begin{array}{ll}
0 & =\displaystyle \int_a^{T}\left[\partial_2 L[x](t)h(t)+\partial_3L[x](t) {_a^C D_t^\alpha}h(t)\right] \,dt+L[x](T) \triangle T\\
   & \quad+\partial_1\phi(T,x(T))\triangle T+\partial_2\phi(T,x(T))[h(T)+x'(T)\triangle T].
\end{array}$$
Integrating by parts (cf. equation \eqref{IP1}), and since $h(a)=0$, we get
\begin{equation}\label{necessarycondition}\begin{array}{ll}
0 & =\displaystyle \int_a^{T}\left[\partial_2 L[x](t)+{_t D_T^\alpha}( \partial_3L[x](t))\right] h(t) \,dt+ \left[ {_t I_T^{1-\alpha}}(\partial_3L[x](t))h(t) \right]_{t=T}  +L[x](T) \triangle T\\
   & \quad+\partial_1\phi(T,x(T))\triangle T+\partial_2\phi(T,x(T))[h(T)+x'(T)\triangle T]\\
  & = \displaystyle \int_a^{T}\left[\partial_2 L[x](t)+{_t D_T^\alpha}( \partial_3L[x](t))\right] h(t) \,dt \\
      & \quad +\triangle T \left[ L[x](T)+ \partial_1\phi(T,x(T)) -x'(T) \left[ {_t I_T^{1-\alpha}}\partial_3L[x](t) \right]_{t=T} \right]\\
  & \quad +  \left[\left[{_t I_T^{1-\alpha}}\partial_3L[x](t)  \right]_{t=T} + \partial_2\phi(T,x(T)) \right] \left[  h(T)+x'(T) \triangle T\right].  \\
      \end{array}\end{equation}
As $h$ and $\triangle T$ are arbitrary we can choose $h(T)=0$ and $\triangle T=0$.  Then, by the
fundamental lemma of the calculus of variations we deduce equation
\eqref{eqeuler}. But if $x$ is a solution of
\eqref{eqeuler}, then the condition
\eqref{necessarycondition} takes the form
\begin{equation}\label{necessarycondition_after}\begin{array}{ll}
0 & =\triangle T \left[ L[x](T)+ \partial_1\phi(T,x(T)) -x'(T) \left[ {_t I_T^{1-\alpha}}\partial_3L[x](t) \right]_{t=T} \right]\\
  & \quad +  \left[\left[{_t I_T^{1-\alpha}}\partial_3L[x](t)  \right]_{t=T} + \partial_2\phi(T,x(T)) \right] \left[  h(T)+x'(T) \triangle T\right].  \\
      \end{array}\end{equation}
Restricting ourselves to those $h$ for which $h(T)=-x'(T) \triangle T$ we get the first equation of \eqref{eqtransv}. Analogously, considering those variations for which $\triangle T=0$ we get the second equation of \eqref{eqtransv}.
\end{proof}

\begin{ex} Let
$$J(x,T)=\int_0^T \left[ t^2-1+( {_a^C D_t^\alpha} x(t))^2 \right]\,dt, \quad T\in[0,10].$$
It easy to verify that a constant function $x(t)=K$ and the end-point $T=1$ satisfies the necessary conditions of optimality of Theorem \ref{Teotransv}, with the value of $K$ being determined by the initial-point $x(0)$.
\end{ex}

In the case when $\alpha$ goes to $1$, by Theorem~\ref{Teotransv} we
obtain the following result.

\begin{col} (\cite{Chachuat}, Theorem 2.24) If $(x,T)$ gives a minimum (or maximum) for
$$J(x,T)=\int_a^T L(t,x(t),x'(t))\,dt+\phi(T,x(T))$$
 on the set
$$\left\{ x\in C^1([a,b]) \, | \, x(a)=x_a \right\}\times[a,b],$$
then $x$ is a solution of the Euler--Lagrange equation
$$\partial_2 L(t,x(t),x'(t))-\frac{d}{dt}\partial_3L(t,x(t),x'(t))=0$$
on the interval $[a,T]$ and satisfies the transversality conditions
$$L(T,x(T),x'(T))+ \partial_1\phi(T,x(T)) -x'(T) \partial_3L(T,x(T),x'(T))=0,$$
$$\partial_3L(T,x(T),x'(T))+ \partial_2\phi(T,x(T))=0.$$
\end{col}

Now we shall rewrite the necessary conditions \eqref{eqtransv} in
terms of the increment on time $\triangle T$ and on the consequent
increment on $x$, $\triangle x_T$. Let us fix $\epsilon=1$ and
consider variation functions $h$ satisfying the additional condition
$h'(T)=0$. Define the total increment by
$$\triangle x_T=(x+h)(T+\triangle T)-x(T).$$
Doing Taylor's expansion up to first order, for a small $\triangle
T$, we have
$$(x+h)(T+\triangle T)-(x+h)(T)=x'(T)\triangle T+O(\triangle T)^2,$$
and so we can write $h(T)$ in terms of $\triangle T$ and $\triangle x_T$:
\begin{equation}\label{h(T)}h(T)=\triangle x_T-x'(T)\triangle T+O(\triangle T)^2.\end{equation}
If $(x,T)$ gives an extremum for $J$, then
$$\partial_2 L[x](t)+{_t D_T^\alpha}( \partial_3L[x](t))=0 \quad \mbox{holds for all } t\in[a,T].$$
Therefore, substituting \eqref{h(T)} into equation \eqref{necessarycondition} we obtain
$$\triangle T \left[ L[x](T)+ \partial_1\phi(T,x(T)) -x'(T)\left[{_t I_T^{1-\alpha}}\partial_3L[x](t)  \right]_{t=T}\right]$$
\begin{equation}\label{necessarycondition2}
\quad +  \triangle x_T \left[ \left[ {_t
I_T^{1-\alpha}}\partial_3L[x](t) \right]_{t=T}+
\partial_2\phi(T,x(T))\right]+   O(\triangle T)^2=0.\end{equation}
We remark that the above equation is evaluated in one single point
$x=T$. Equation \eqref{necessarycondition2} replaces the missing
terminal condition in the problem.

Let us consider five particular cases of constraints that can be specified in the optimization problem.

\textbf{A. Vertical terminal line}

In this case the upper bound $T$ is fixed and consequently the
variation $\triangle T=0$. Therefore, equation
\eqref{necessarycondition2} becomes
$$\triangle x_T \left[ \left[ {_t I_T^{1-\alpha}}\partial_3L[x](t) \right]_{t=T}+ \partial_2\phi(T,x(T))\right]=0,$$
and by the arbitrariness of $\triangle x_T$, we deduce
$$ \left[ {_t I_T^{1-\alpha}}\partial_3L[x](t) \right]_{t=T}+ \partial_2\phi(T,x(T))=0.$$
When $\phi\equiv 0$, we get the natural boundary condition as
obtained in \cite{AGRA1}, equation (25):
$$ \left[ {_t I_T^{1-\alpha}}\partial_3L[x](t) \right]_{t=T}=0.$$
For $\phi\equiv 0$ and $\alpha\rightarrow 1$, we get the second
equation of \eqref{naturalBound}.

\textbf{B. Horizontal terminal line}

Now we have $\triangle x_T=0$ but $\triangle T$ is arbitrary. Hence,
equation \eqref{necessarycondition2} implies
$$L[x](T)+ \partial_1\phi(T,x(T))-x'(T) \left[{_t I_T^{1-\alpha}}\partial_3L[x](t)  \right]_{t=T}=0.$$
When $\phi\equiv 0$ and $\alpha\rightarrow 1$, we get equation
(3.11) of \cite{Chiang}. It is worth pointing out that the integer
case has an economic interpretation (see explanation in
\cite{Chiang}, pags. 63--64).

\textbf{C. Terminal curve}

In this case the terminal point is described by a given curve
$\psi$, in the sense that $x(T)=\psi(T)$, where
$\psi:[a,b]\to\mathbb R$ is a prescribed differentiable curve. For a
small $\triangle T$, from Taylor's formula, one has
\begin{equation}\label{from:Tay}
\begin{array}{ll}
\triangle x_T & = \psi(T+\triangle T)-\psi(T)\\
              & = \psi'(T) \triangle T+ O(\triangle T)^2.\\
\end{array}
\end{equation}
Substituting \eqref{from:Tay} into \eqref{necessarycondition2}
yields
$$(\psi'(T)-  x'(T) ) \left\{ \left[ {_t I_T^{1-\alpha}}\partial_3L[x](t) \right]_{t=T}+ \partial_2\phi(T,x(T))\right\} + L[x](T)+ \partial_1\phi(T,x(T))=0.$$
For $\phi\equiv 0$, we obtain equation (29) of \cite{AGRA2}. For
$\phi\equiv 0$ and $\alpha\rightarrow 1$, we have equation (3.12) of
\cite{Chiang}.

\textbf{D. Truncated vertical terminal line}

Now we consider the case where $\triangle T=0$ and $x(T)\geq
x_{min}$. Here $x_{min}$ is a minimum permissible level of $x$. By the Kuhn-Tucker
conditions, we obtain
$$\left[ {_t I_T^{1-\alpha}}\partial_3L[x](t) \right]_{t=T}+ \partial_2\phi(T,x(T))\leq 0,\quad x(T)\geq
x_{min},$$ $$(x(T)-x_{min})\left\{ \left[ {_t
I_T^{1-\alpha}}\partial_3L[x](t) \right]_{t=T}+
\partial_2\phi(T,x(T))\right\}=0$$
for maximization problem;
and
$$\left[ {_t I_T^{1-\alpha}}\partial_3L[x](t) \right]_{t=T}+ \partial_2\phi(T,x(T))\geq 0,\quad x(T)\geq
x_{min},$$ $$(x(T)-x_{min})\left\{\left[ {_t
I_T^{1-\alpha}}\partial_3L[x](t) \right]_{t=T}+
\partial_2\phi(T,x(T))\right\}=0$$
for minimization problem.  If $\phi\equiv 0$ and
$\alpha\rightarrow 1$, we get equations (3.17) and (3.17') of
\cite{Chiang}.

\textbf{E. Truncated horizontal terminal line}

In this situation, the constraints are $\triangle x_T=0$ and $T\leq T_{max}$. Therefore, as in the previous case, we obtain:
$$ L[x](T)+ \partial_1\phi(T,x(T)) -x'(T)\left[{_t I_T^{1-\alpha}}\partial_3L[x](t)  \right]_{t=T}\geq 0,\quad T\leq T_{max},$$
$$(T-T_{max})\left[ L[x](T)+ \partial_1\phi(T,x(T)) -x'(T)\left[{_t I_T^{1-\alpha}}\partial_3L[x](t)  \right]_{t=T}\right]=0$$
for maximization problem; and
$$ L[x](T)+ \partial_1\phi(T,x(T)) -x'(T)\left[{_t I_T^{1-\alpha}}\partial_3L[x](t)  \right]_{t=T}\leq 0,\quad T\leq T_{max},$$
$$(T-T_{max})\left[ L[x](T)+ \partial_1\phi(T,x(T)) -x'(T)\left[{_t I_T^{1-\alpha}}\partial_3L[x](t)  \right]_{t=T}\right]=0$$
for minimization problem.
Observe that for $\phi\equiv 0$ and $\alpha\rightarrow 1$, we get equation (3.18)
and (3.18') of \cite{Chiang}.


\section{Transversality conditions II}
\label{sec4}

In this section we consider the following variational problem:
\begin{equation}
\label{problem}
\begin{gathered}
J(x,T)=\int_{a}^{T} L\{x,\varphi\}(t,T)dt \longrightarrow \textrm{extr}\\
(x,T) \in \mathrm{D}\\
x(T)=\varphi(T)
\end{gathered}
\end{equation}
where $L:[a,b]\times \mathbb R^3\to\mathbb R$ and
$\varphi:[a,b]\to\mathbb R$ are at least of class $C^1$. Here
$x=\varphi (t)$ is a specified curve.

\begin{thm}\label{Teotransv_2} Suppose that $(x,T)$ is a solution to problem \eqref{problem}.
Then $x$ is a solution of the fractional Euler--Lagrange equation
\begin{equation}\label{E-L}
\partial_2
L\{x,\varphi\}(t,T)+{_t D_T^\alpha}(
\partial_3L\{x,\varphi\}(t,T))=0
\end{equation}
on the interval $[a,T]$, and satisfies the transversality condition
\begin{equation}\label{eqtransv_2}
(\varphi'(T)-x'(T)) \left[ {_t
I_T^{1-\alpha}}\partial_3L\{x,\varphi\}(t,T)
\right]_{t=T}+\varphi'(T)\int_{a}^{T}\partial_4L\{x,\varphi\}(t,T)dt+L\{x,\varphi\}(T,T)=0.
\end{equation}
\end{thm}

\begin{proof}
Suppose that $(x,T)$ is a solution to problem \eqref{problem} and
consider the value of $J$ at an admissible variation $(
x(t)+\epsilon h(t),T+\epsilon \triangle T)$, where $\epsilon\in
\mathbb R$ is a small parameter, $\triangle T\in\mathbb R$, and
$h\in C^1([a,b])$ with $h(a)=0$. Let
$$j(\epsilon)=J(x+\epsilon h,T+\epsilon \triangle
T)=\displaystyle\int_a^{T+\epsilon \triangle T} L\{x+\epsilon
h,\varphi\}(t,T+\epsilon \triangle T)dt.$$ Then, a necessary
condition for $(x,T)$ to be a solution to problem \eqref{problem} is
given by
$$\begin{array}{ll} j'(0)=0 \Leftrightarrow
 &\displaystyle\int_a^{T}\left[\partial_2
L\{x,\varphi\}(t,T)h(t)+\partial_3L\{x,\varphi\}(t,T) {_a^C
D_t^\alpha}h(t) +\partial_4L\{x,\varphi\}(t,T)\varphi'(T)\triangle
T\right]dt\\
&+L\{x,\varphi\}(T,T)\triangle T=0.
\end{array}$$
Integrating by parts (cf. equation \eqref{IP1}), and since $h(a)=0$, we get
\begin{equation}\label{necessarycondition_2}\begin{array}{ll}
0 & =\displaystyle \int_a^{T}\left[\partial_2
L\{x,\varphi\}(t,T)+{_t D_T^\alpha}(
\partial_3L\{x,\varphi\}(t,T))\right] h(t) \,dt+ \left[ {_t
I_T^{1-\alpha}}(\partial_3L\{x,\varphi\}(t,T))h(t) \right]_{t=T}\\
 &+L\{x,\varphi\}(T,T)\triangle T+\int_a^{T}\left[\partial_4
L\{x,\varphi\}(t,T)\varphi'(T)\triangle T\right]dt.
\end{array}\end{equation}
As $h$ and $\triangle T$ are arbitrary, first we consider $h$ and
$\triangle T$ such that $h(T)=0$ and $\triangle T=0$. Then, by the
fundamental lemma of the calculus of variations we deduce equation
\eqref{E-L}. Therefore, in order for $(x,T)$ to
be a solution to problem \eqref{problem}, $x$ must be a solution of
the fractional Euler--Lagrange equation. But if $x$ is a solution of
\eqref{E-L}, the first integral in expression
\eqref{necessarycondition_2} vanishes, and then the condition
\eqref{necessarycondition_2} takes the form
\begin{equation}\label{trans}
\left[ {_t I_T^{1-\alpha}}(\partial_3L\{x,\varphi\}(t,T))h(t)
\right]_{t=T}+L\{x,\varphi\}(T,T)\triangle
T+\int_a^{T}\left[\partial_4 L\{x,\varphi\}(t,T)\varphi'(T)\triangle
T\right]dt=0.
\end{equation}
Since the right hand point of $x$ lies on the curve $z=\varphi(t)$,
we have
$$x(T+\epsilon \triangle T)+\epsilon h(T+\epsilon \triangle
T)=\varphi(T+\epsilon \triangle T).$$ Hence, differentiating
with respect $\epsilon$ and setting $\epsilon=0$ we get
\begin{equation}\label{curve}
h(T)=(\varphi'(T)-x'(T))\triangle T.
\end{equation}
Substituting \eqref{curve} into \eqref{trans} yields
\begin{multline*}
\triangle T(\varphi'(T)-x'(T)) \left[ {_t
I_T^{1-\alpha}}\partial_3L\{x,\varphi\}(t,T)
\right]_{t=T}\\
+\varphi'(T)\triangle
T\int_{a}^{T}\partial_4L\{x,\varphi\}(t,T)dt+L\{x,\varphi\}(T,T)\triangle
T=0.
\end{multline*}
Since $\triangle T$ can take any value, we
obtain condition \eqref{eqtransv_2}.
\end{proof}

In the case when $\alpha$ goes to $1$, by Theorem~\ref{Teotransv_2}
we obtain the following result.

\begin{col} If $(x,T)$ gives an extremum for
$$J(x,T)=\int_a^T L(t,x(t),x'(t),\varphi(T))\,dt$$
 on the set
$$\left\{ x\in C^1([a,b]) \, | \, x(a)=x_a \right\}\times[a,b],$$
then $x$ is a solution of the Euler--Lagrange equation
$$\partial_2 L(t,x(t),x'(t),\varphi(T))-\frac{d}{dt}\partial_3L(t,x(t),x'(t),\varphi(T))=0, \quad \mbox{for all } t\in[a,T],$$
and satisfies the following transversality condition
$$(\varphi'(T)-x'(T))\partial_3L(T,x(T),x'(T),\varphi(T))+\varphi'(T)\int_a^T \partial_4L(t,x(t),x'(t),\varphi(T))\,dt$$
$$+L(T,x(T),x'(T),\varphi(T))=0.$$
\end{col}

\begin{ex}
Consider the following problem
\begin{equation}
\label{EX}
\begin{gathered}
J(x,T)=\int_0^T \left[({^C_0 D_t^{\alpha}}x(t))^2+\varphi^2(T)\right]dt \longrightarrow \min\\
x(T)=\varphi(T)=T,\\
x(0)=1.
\end{gathered}
\end{equation}
For this problem, the fractional Euler--Lagrange equation and the
transversality condition (see Theorem~\ref{Teotransv_2}) are given,
respectively, by
\begin{equation}\label{Ex:el}
{_tD_T^{\alpha}}\left({^C_0 D_t^{\alpha}} x(t)\right)=0,
\end{equation}
\begin{equation}\label{Ex:tra}
2(1-x'(T))\left[{_t I_T^{1-\alpha}}{^C_0 D_t^{\alpha}}
x(t)\right]|_{t=T} +2T^2+({^C_0 D_t^{\alpha}}x(T))^2+T^2=0.
\end{equation}
Note that it is a difficult task to solve the above fractional
equations. For $0<\alpha <1$ a numerical or direct method should be
used \cite{Ric}. When $\alpha$ goes to $1$, problem \eqref{EX}
becomes
\begin{equation}
\label{EX:1}
\begin{gathered}
J(x,T)=\int_0^T \left[x'(t))^2+\varphi^2(T)\right]dt \longrightarrow \min\\
x(T)=\varphi(T)=T,\\
x(0)=1.
\end{gathered}
\end{equation}
and equations \eqref{Ex:el}--\eqref{Ex:tra} are replaced by
\begin{equation}\label{Ex:el:1}
x''(t)=0,
\end{equation}
\begin{equation}\label{Ex:tra:1}
2(1-x'(T))x'(T) +2T^2+(x'(T))^2+T^2=0.
\end{equation}
Solving equations \eqref{Ex:el:1} and \eqref{Ex:tra:1} we obtain
that
\begin{equation*}
\tilde{x}(t)=\frac{\sqrt{-6+6\sqrt{13}}-6}{\sqrt{-6+6\sqrt{13}}}t+1,
\quad T=\frac{1}{6}\sqrt{-6+6\sqrt{13}}
\end{equation*}
is a candidate for minimizer to problem \eqref{EX:1}.
\end{ex}


\section{Infinite horizon fractional variational problems}
\label{sec5}

Starting with the Ramsey pioneering work \cite{ram}, infinite horizon
variational problems have been widely used in economics (see,
\textrm{e.g.}, \cite{Brock,Chiang,Kami,Mal,Nitta} and the references
therein). One may assume that, due to some constraints of economical
nature, the infinite horizon variational problem does not depend on
the usual derivative but on the left Caputo fractional derivative.
In this condition one has to consider the following problem:
\begin{equation}
\label{inf}
\begin{gathered}
J(x)=\int_{a}^{+\infty} L[x](t)dt \longrightarrow \textrm{max}\\
x\in \mathrm{D}_\infty,
\end{gathered}
\end{equation}
where $L:[a,+\infty]\times \mathbb R^2\to\mathbb R$ is at least of
class $C^1$ and
$$\mathrm{D}_\infty=\left\{ x\in C^1([a,+\infty]) \, | \, x(a)=x_a \right\}.$$
The integral in problem \eqref{inf} may not converge.  In the case where the integral diverges, there may exist more than one path that yields an infinite value for the objective functional and it would be difficult to determine which among these paths is optimal. To handle this and similar situations
in a rigorous way, several alternative definitions of optimality for problems with unbounded time
domain have been proposed in the literature (see,
\textrm{e.g.}, \cite{Brock,gale,Kami}).  Here, we follow
Brock's notion of optimality \cite{Brock}, \textrm{i.e.}, a function
$x\in \mathrm{D}_\infty$ is said to be weakly maximal to problem
\eqref{inf} if
$$\lim_{T\rightarrow+\infty}\inf_{T' \geq T}\int_{a}^{T'}[L[\overline x](t)-L[x](t)] \, dt \leq 0,$$
for all $\overline x \in \mathrm{D}_\infty$.\\
Using similar approach as in \cite{Mal,Nitta}, we obtain the
transversality condition for infinite horizon fractional variational
problems. First, a lemma that we will use in the proof of
Theorem~\ref{Thconinf}.

\begin{lem}
\label{fun} If $g$ is continuous on $[a,+\infty)$ and
$$\lim_{T\rightarrow+\infty}\inf_{T'\geq T}\int_{a}^{T'} g(t) h(t)\,dt=0$$
for all continuous functions $h:[a,+\infty)\to\mathbb R$ with $h(a)=0$, then $g(t)=0$, for all $t \geq a$.
\end{lem}
\begin{proof}
Can be done in a similar way as the proof of the standard
fundamental lemma of the calculus of variations (see, \textrm{e.g.},
\cite{Brunt}).
\end{proof}

\begin{thm} \label{Thconinf}
Suppose that $x$ is a weakly maximal to problem
\eqref{inf}. Let $h \in C^1([a,+\infty])$ and $\epsilon \in\mathbb
R$. Define
$$\begin{array}{lcl}
 A(\epsilon, T')& = & \displaystyle \int_{a}^{T'} \frac{L[x+\epsilon h](t)- L[x](t)}{\epsilon}dt;\\
& & \\
V(\epsilon, T)& = & \displaystyle \inf_{T' \geq T}\int_{a}^{T'} [L[x+\epsilon h](t)-L[x](t)]dt;\\
& & \\
V(\epsilon)&= & \displaystyle \lim_{T\rightarrow+\infty} V(\epsilon,
T).
\end{array}$$
Suppose that
\begin{enumerate}
\item $\displaystyle \lim_{\epsilon \rightarrow 0} \frac{V(\epsilon, T) }{\epsilon}$ exists for all $T$;
\item $\displaystyle \lim_{T\rightarrow+\infty}\frac{V(\epsilon, T) }{\epsilon}$ exists uniformly for $\epsilon$;
\item For every $T'> a$, $T > a$, and $\epsilon\in \mathbb{R}\setminus\{0\}$, there exists a sequence $\left(A(\epsilon, T'_n)\right)_{n \in \mathbb{N}}$
such that
$$\displaystyle \lim_{n \rightarrow +\infty} A(\epsilon, T'_n)= \displaystyle \inf_{T'\geq T} A(\epsilon, T')$$
uniformly for $\epsilon$.
\end{enumerate}
Then $x$ is a solution of the fractional Euler–-Lagrange equation
\begin{equation*}
\partial_2 L[x](t)+ {_tD_{T}^\alpha}( \partial_3 L[x](t))=0
\end{equation*}
for all $t\in[a,+\infty)$, and for all $T>t$. Moreover it satisfies
the transversality condition
\begin{equation*}
 \lim_{T\rightarrow+\infty} \inf_{T'\geq T}
{_tI^{1-\alpha}_{T'}}  \partial_3 L[x](t) ]_{t=T'}  x(T')=0.
\end{equation*}
\end{thm}

\begin{proof} Observe that $V(\epsilon) \leq 0$ for every $\epsilon \in \mathbb{R}$, and $V(0)=0$. Thus $V'(0)=0$, i.e.,
$$\begin{array}{lcl}
0 & = & \displaystyle \lim_{\epsilon \rightarrow 0} \frac{V(\epsilon)}{\epsilon}\\
& = & \displaystyle  \lim_{T\rightarrow+\infty}  \inf_{T' \geq T} \displaystyle \int_{a}^{T'} \lim_{\epsilon \rightarrow0} \frac{L[x+\epsilon h](t)- L[x](t)}{\epsilon}dt\\
& = & \displaystyle  \lim_{T\rightarrow+\infty}  \inf_{T' \geq T} \displaystyle \int_{a}^{T'} \left[\partial_2 L[x](t)h(t)+ \partial_3 L[x](t){_a^C D_t^\alpha}h(t)\right]  dt.\\
\end{array}$$
Thus, integrating by parts and since $h(a)=0$, we get
\begin{equation}\label{eq_1}
 \lim_{T\rightarrow+\infty}  \inf_{T' \geq T}\left\{ \displaystyle \int_{a}^{T'} [ \partial_2L[x](t)+ {_tD_{T'}^\alpha}( \partial_3 L[x](t))]h(t)dt +  [ {_tI^{1-\alpha}_{T'}}  \partial_3 L[x](t) h(t) ]_{t=T'} \right\}=0.
\end{equation}
By the arbitrariness of $h$, we may assume that $h(T')=0$ and so
\begin{equation}\label{eq_3}\lim_{T\rightarrow+\infty}  \inf_{T' \geq T} \displaystyle \int_{a}^{T'} [ \partial_2L[x](t)+ {_tD_{T'}^\alpha}( \partial_3 L[x](t))]h(t)dt =0.\end{equation}
Applying Lemma~\ref{fun} we obtain
$$\partial_2 L[x](t)+ {_tD_{T}^\alpha}( \partial_3 L[x](t))=0$$
for all $t\in[a,+\infty)$, and for all $T>t$. Substituting equation
\eqref{eq_3} into equation \eqref{eq_1} we get
\begin{equation}
\label{eq_4} \displaystyle \lim_{T\rightarrow+\infty} \inf_{T'\geq
T}[ {_tI^{1-\alpha}_{T'}}  \partial_3 L[x](t) ]_{t=T'}  h(T')= 0.
\end{equation}
To eliminate $h$ from equation \eqref{eq_4}, consider the particular
case $h(t)= \chi (t)x(t)$, where $\chi:[a,+\infty)\to\mathbb R$ is a
function of class $C^1$ such that $\chi(a)=0$ and
$\chi(t)=const\not=0$, for all $t>t_0$, for some $t_0>a$. We deduce
that
$$\displaystyle \lim_{T\rightarrow+\infty} \inf_{T'\geq T}[ {_tI^{1-\alpha}_{T'}}  \partial_3 L[x](t) ]_{t=T'}  x(T')= 0.$$
\end{proof}


\section{Conclusion}
\label{sec:conc}

Transversality conditions are optimality conditions that are used along with Euler--Lagrange equations in order to find the optimal
paths (plans, programs, trajectories, \textrm{etc}) of dynamical models. The importance of such conditions is well known in economics models and other phenomena whose effects can be spread along time, \textrm{e.g.}, radioactive, pollution. In this paper we have given in the compact form transversality
conditions for fractional variational problems with the Caputo derivative. The fractional variational theory is only 15 years old so there are many unanswered questions. For instance, the question of existence of solutions to fractional variational problems
is a complete open area of research (more about this important issue the reader can find in \cite{cesari,comTatiana:Basia,torres2004,young}). Other interesting open question is about the convergence of the objective functional $J(x)=\int_{a}^{+\infty} L[x](t)dt$. This issue should be treat carefully, especially as the functional depends on fractional type of derivatives.
Generally, in order to solve fractional Euler--Lagrange differential equations
and apply transversality conditions one needs to use numerical methods. Some progress is being made already on the numerical methods for fractional variational problems \cite{agrawal:et:al:2012,Shakoor:01,MyID:221,MyID:225} but there is still a long way to go.
We believe that all pointed issues will be considered in a forthcoming papers.


\section*{Acknowledgements}

This work is supported by {\it FEDER} funds through
{\it COMPETE} --- Operational Programme Factors of Competitiveness
(``Programa Operacional Factores de Competitividade'')
and by Portuguese funds through the
{\it Center for Research and Development
in Mathematics and Applications} (University of Aveiro)
and the Portuguese Foundation for Science and Technology
(``FCT --- Funda\c{c}\~{a}o para a Ci\^{e}ncia e a Tecnologia''),
within project PEst-C/MAT/UI4106/2011
with COMPETE number FCOMP-01-0124-FEDER-022690.
 Agnieszka B. Malinowska is supported by Bia{\l}ystok University of Technology grant S/WI/02/2011.



\end{document}